\documentclass[a4paper,12pt]{article}

\newenvironment{proof}{\noindent {\bf Proof:}}{\hfill $\Box$}

\newtheorem{theorem}{Theorem}
\newtheorem{lemma}{Lemma}
\newtheorem{proposition}{Proposition}

\newtheorem{definition}{Definition}

\newtheorem{remark}{Remark}

\usepackage{booktabs}
\usepackage{color}

\usepackage{algorithm}
\usepackage[noend]{algpseudocode}
\makeatletter
\def\BState{\State\hskip-\ALG@thistlm}
\makeatother
\usepackage{algorithmicx}

\usepackage{amsmath}
\usepackage{amssymb}
\usepackage{amsfonts}
\usepackage{hyperref}

\usepackage[normalem]{ulem}
\usepackage[dvipsnames]{xcolor}

\usepackage{tikz}

\newcommand{\dist}{\mathrm{dist}}
\newcommand{\supp}{\mathrm{supp}}
\newcommand{\A}{\mathcal{A}}
\newcommand{\N}{\mathbb{N}}
\newcommand{\R}{\mathbb{R}}
\newcommand{\C}{\mathcal{C}}

\newcommand{\one}{1}

\title{\bf Converging outer approximations to global attractors using semidefinite programming}
\date{}

\begin{document}

\maketitle

\begin{center}
\author{Corbinian Schlosser$^{1}$},
\author{Milan Korda$^{\text{1},\text{2}}$}
\end{center}

\footnotetext[1]{CNRS; LAAS; 7 avenue du colonel Roche, F-31400 Toulouse; France. {\tt cschlosser@laas.fr, korda@laas.fr}}
\footnotetext[2]{Faculty of Electrical Engineering, Czech Technical University in Prague,
Technick\'a 2, CZ-16626 Prague, Czech Republic.}

\begin{abstract}
This paper develops a method for obtaining guaranteed outer approximations for global attractors of continuous and discrete time nonlinear dynamical systems. The method is based on a hierarchy of semidefinite programming problems of increasing size with guaranteed convergence to the global attractor. We follow an established line of reasoning, where we first characterize the global attractor via an infinite dimensional linear programming problem (LP) in the space of Borel measures. We then turn to the dual LP acting on the space of continuous functions. Feasible solutinos of the dual LP provide guaranteed outer approximations to the global attractor. For systems with polynomial dynamics, a hierarchy of finite-dimensional sum-of-squares tightenings of the dual LP provides a sequence of outer approximations to the global attractor with guaranteed convergence in the sense of volume discrepancy tending to zero. The method is very simple to use and based purely on convex optimization. Numerical examples with the code available online demonstrate the method.
\end{abstract}

\begin{flushleft}\small
\end{flushleft}

\section{Introduction}
The global attractors, i.e., sets to which all trajectories converge asymptotically, are typical and important objects in the analysis of dynamical systems. These sets have many intriguing and desirable properties; for instance, all trajectories can be approximated by trajectories on the global attractor asymptotically (see~\cite{c28neu} or Remark \ref{RemarkAttractorProperties} 2.). In terms of applications for control, the global attractor determines the set of all states to which trajectories asymptotically converge under a given feedback control law. This set may be a single point in the state-space or a more complicated object such as a periodic orbit, either by design or due to nonlinearities (e.g., dead-zone, backlash, quantization etc.) that prevent stabilization to a single point~\cite{c3neu}. Localizing this set in the state-space is therefore important for determining the asymptotic performance of the controller.

We approach the problem of state-space localization of the attractor from a convex optimization viewpoint, utilizing the so-called occupation measures (primal problem) or continuous functions (dual problem) defined on the state space; both of these problems are infinite-dimensional linear programs (LPs).  The approach builds on the work~\cite{c10neu} that characterizes and approximates computationally the maximum positively (or control) invariant set; the second ingredient of our approach is the characterization of the global attractor as the largest set that is invariant both forward and backward in time. This leads to an infinite dimensional LP characterization of the global attractor. In order to solve the resulting infinite dimensional linear program we apply the moment-sum-of-squares hierarchy (see, e.g., \cite{c14neu} and \cite{c16neu}), leading to a sequence of finite-dimensional semidefinite programming problems (SDPs). Solutions to these SDPs provide guaranteed enclosures of the global attractor which are proven to converge in terms of volume discrepancy tending to zero. We treat both the continuous and discrete time systems, with the latter being more challenging theoretically due to time reversal issues. The approach is very simple to use, with the outer approximations obtained  from the solution to a single convex SDP without any iteration or complicated initialization.

Historically, the idea of transforming various problems from nonlinear dynamical systems and control into infinite-dimensional LPs dates back at least to the work \cite{c29neu} dedicated to optimal control. Solving these problems by a hierarchy of semidefinite problems (SDPs) with proven convergence was proposed in \cite{c15neu}, although SDP approximations to infinite-dimensional LPs were used already in~\cite{c26neu} for the problem of global stabilization. Since then, this approach was used to tackle a number of problems, including the region of attraction~\cite{c8neu}, maximum (control) invariant and reachable sets~\cite{c10neu,c20neu} or, more recently, analysis and control of nonlinear partial differential equations~\cite{c6neu,c11neu,c21neu}, to name just a few. Closest to our work from this line of research is~\cite{c5neu}, treating the problem of estimating the maximum of a given function on the attractor.

A classical approach to attractors is via Lyapunov functions (see, e.g., \cite{c4neu}). There are many ans\"atze for constructing Lyapunov functions. Some use partial differential equations, like Zubov's equation, whose solutions give Lyapunov functions. A partial differential inequality of a Lyapunov type also occurs in our approach and partial differential equation appears in one of our proofs as a technical tool. Sum-of-squares methods have been systematically used for stability analysis of given fixed points or attractors through searching for polynomial Lyapunov functions (e.g.,~\cite[chapter 7]{c23neu}). However, the goal of this work is not stability analysis of a given attractor but rather localization of an unknown attractor in the state-space. Other possible approaches for approximating global attractors are, for example, the set oriented methods~\cite{c1neu}.

\section{Notations}
The natural numbers are with zero included and denoted by $\N$. The function $\dist(\cdot,K)$ denotes the distance function to $K$ and $\dist(K_1,K_2)$ denotes the Hausdorff distance of two subsets of $\R^n$. The set of positive measures on $\R^n$ supported on $X$ are denoted by $M(X)$. The support of an element $\mu \in M(X)$ is denoted by $\supp(\mu)$. The space of continuous functions on $X$ is denoted by $\C(X)$ and the space of continuously differentiable functions on $\R^n$ by $\C^1(\R^n)$.
The Lebesgue measure will always be denoted by $\lambda$ and its restriction to $X$ by $\lambda\big|_X$. The indicator function for a set $C \subset X$ is denoted by $\mathbb{I}_C:X \rightarrow \R$ and is given by $\mathbb{I}_C(x) = 1$ if $x \in C$ and $\mathbb{I}_C(x) = 0$ else. The ring of multivariate polynomials in variables $x = (x_1,\ldots,x_n)$ is denoted by $\R[x] = \R[x_1,\ldots,x_n]$ and for $k \in \N$ the ring of multivariate polynomials of total degree at most $k$ is denoted by $\R[x]_k$. We will denote the open ball of radius $r$ with respect to the euclidean metric by $B_r(0)$.

\section{Setting and preliminary definitions}
Let $X \subset \R^n$ be compact and $f:\R^n \rightarrow \R^n$ have locally Lipschitz continuous derivative. Let $\varphi_{t}(x_0)$ be the solution at time $t$ of
\begin{equation}\label{EqnODE}
	\dot{x} = f(x), \;\; x(0) = x_0 \in \R^n
\end{equation} 
The set of all initial values $x_0 \in X$ such that $\varphi_t(x_0)  \in X$ for all $t \in \R_+$ is called the maximal positively invariant (MPI) set and will be denoted by $M_+$ in the following.

\begin{definition}[Global attractor]
	A compact set $\A \subset X$ is called a global attractor if it is minimal uniformly attracting, i.e., it is the smallest compact set $\A$ such that
	\[
	\lim_{t\to\infty} \dist(\varphi_t(M_+),\A) = 0,
	\]
	where $M_+$ is the MPI set.
\end{definition}
\begin{remark} Initial conditions of trajectories that leave $X$ at some point (even though they may return to $X$) are not part of the global attractor $\A$. However, all trajectories that stay in $X$ for all positive times converge to $\A$ uniformly as $t \rightarrow \infty$. A global attractor is maximal invariant (forward and backward)~\cite{c28neu}, i.e., for all $t \in \R_+$ the flow $\varphi_t(x)$ is defined for all $x \in \A$ and we have $\varphi_t(\A) = \A$. Further, if a global attractor exists it is unique. This is an immediate consequence of being minimal attractive~\cite{c28neu}.
\end{remark}

It is worth mentioning some differences between the global attractor and the weak attractor $\A_w$ given by the smallest closed set that attracts all trajectories, i.e. for all $x \in M_+$
\begin{equation}\label{EqnWeakAttractor}
	\dist(\varphi_t(x),\A_w) \rightarrow 0 \text{ as } t \rightarrow \infty.
\end{equation}
Obviously, we have $\A_w \subset \A$ but those sets can differ significantly. A typical example that illustrates many differences is a dynamical system that is given by a heteroclinic orbit $\Gamma$, connecting an unstable equilibrium point $x_0$ with a stable equilibrium point $x_1$, e.g.,
\[
\dot{x} = (x+1)(1-x), \;\; x\in \R.
\]

 The weak attractor is given by $\A_w = \{x_0,x_1\} = \{-1,1\}$ while the global attractor is given by $\A = \Gamma = [-1,1] $. The latter is due to $\varphi_t(\Gamma) = \Gamma $ because $x_0 = -1$ is an unstable equilibrium point. The topological differences in this case are obvious. But one might wonder why it is of interest to work with the much larger set $\A$ instead of $\A_w$. This question is of course very reasonable and has good answers. We refer to \cite{c28neu} for an exhaustive treatment of global attractor and state some results from there here.
 
\begin{remark}\label{RemarkAttractorProperties}
The global attractor has the following properties
	\begin{enumerate}
		\item The global attractor is Lyapunov stable.\\
		An example for an unstable weak attractor is a heteroclinic orbit or Vinograd's example~\cite[p.\,115]{c22neu}.
		\item ``The global attractor approximates trajectories" \cite{c28neu}: For all $x \in M_+$, $T> 0$ and $\varepsilon > 0$ there exists a $t_0 = t_0(T,\varepsilon) \in \R_+$ (independent of $x$!) and $x_0 = x_0(x,T,\varepsilon) \in \A$ such that
		\begin{equation*}
			\|\varphi_t(x)-\varphi_t(x_0)\|_2^2< \varepsilon \text{ for all } t \in [t_0,t_0 + T]. 
		\end{equation*}
		Weak attractors have a similar property but in case of a weak attractor, the time $t_0$ can not be chosen uniformly in $x$, i.e. $t_0$ depends on $x$ in general. An example is given by a simplified version of the system from \cite{c9neu} p. 287 and \cite{c13neu} given by
		\begin{align*}
			\dot{\theta} = \sin(\theta)^2
		\end{align*}
		on the set $X = [0,2\pi)$ identified with the unit circle. The weak attractor is given by the equilibrium points, i.e. $\A_w = \{0,\pi\}$. Both equilibrium points are unstable and hence $t_0$ can not be chosen uniformly in the initial value.
		\item The global attractor is upper semicontinuous (see~\cite{c28neu} for the result and related definitions). That means small changes in the vector field can not cause a drastic increase in the global attractor.\\
		An example of a weak attractor that explodes by a small change of the vector field can be found in~\cite{c28neu} p. 267, where a spiraling trajectory turns into a periodic orbit.
		\item There exists a smooth Lyapunov function $V:M_+ \rightarrow \R$ such that $\frac{\mathrm{d}}{\mathrm{dt}}V(\varphi_t(x)) < 0$ for all $x \in M_+ \setminus V^{-1}(\{0\})$ and $V^{-1}(\{0\}) = \A$.\\
		Since the existence of such a Lyapunov function implies stability for the set $V^{-1}(\{0\})$ a weak attractor that is not stable (as in the first statement) does not allow such a Lyapunov function.
		\item If there exists a strict Lyapunov function $V$, i.e. $V$ is strictly decreasing in time except in equilibrium points, then by Lasalle's invariance principle, $\A$ is given by the unstable manifold of all equilibrium points, i.e. the points that are repelled by the equilibrium points. While the weak attractor is given by just the equilibrium points.
	\end{enumerate}

\end{remark}

The last theorem should not be understood as critic of weak attractors. Weak attractors still have interesting and good properties and tend to be much smaller sets than the global attractors. If the notion of a weak attractor fits the application this is definitely desirable. The purpose of stating the above theorem is to point out why we have a special interest in global attractors.

It seems reasonable that the global attractor is contained in the MPI set $M_+$ because it only cares about trajectories in the closed set $M_+$. On the other hand for any $x \in M_+$ such that $\varphi_{-t}(x) \in X$ for all $t \in \R_+$ (i.e. solving the differential equation backward in time) we know that $x = \varphi_t(\varphi_{-t}(x))$ has to be close to the global attractor for $t \in \R_+$ large enough (because $\A$ is uniformly attracting). This motivates point~3) in the following theorem but also introducing the maximal negatively invariant (MNI) set $M_-$ that consists of all points $x \in X$ such that $\varphi_{-t}(x) \in X$ for all $t \in \R_+$. In other words, the MNI set is the MPI set for the time reversed differential equation (i.e., for the ODE $\dot{x} = -f(x)$).\\
We will recall the following theorem on the existence of global attractors and connect the global attractor to $M_+$ and~$M_-$.
 
\begin{theorem}\label{ThmExistenceGlobalAttractor} Assume $X \subset \R^n$ is compact and $f:X\rightarrow \R^n$ is (locally) Lipschitz then\\
\vspace{-4mm}
	\begin{enumerate}
		\item The global attractor $\A$ exists.
		\item The dynamical system is invertible on the global attractor, i.e. for all $x \in \A$ we can define $\varphi_t(x)$ for all $t \in \R$ and $\varphi$ is continuous in $(t,x)$.
		\item The global attractor consists of all orbits of the solutions that exist for all times $t \in \R$.
		\item The global attractor is the maximal set that is positively invariant forward and backward in time, i.e. $\A = M_+ \cap M_-$.
	\end{enumerate}
\end{theorem}
 
\begin{proof} A very good reference for the above result is \cite{c28neu} chapter 10 where the first three statements can be found. The last statement is just a reformulation of the third statement. 
\end{proof}

\section{A linear program for global attractors}
In \cite{c10neu} a linear program (LP) for approximating the maximal positively invariant set was presented. Theorem \ref{ThmExistenceGlobalAttractor} shows that the global attractor is characterized as the largest set that is positively invariant forward and backward in time. Hence combining the LP for the maximal positively invariant set (in forward time) with the same LP in reversed time direction gives the global attractor. This is the LP we will derive next. 
For an initial measure $\mu_0 \in M(X)$ and a discount factor $\beta > 0$ let $\mu$ be the discounted occupation measure, i.e., for a measurable set $C\subset X$ the value $\mu(C)$ is defined by
\begin{equation}\label{DefOccMeasure}
 \mu(C) := \int\limits_X\int\limits_0^\infty e^{-\beta t} \mathbb{I}_{C}(\varphi_t(x)) \; dt \; d\mu_0(x).
\end{equation}
Then $\mu$ is a well defined measure on $X$ that measures the discounted average time spent in $C$ where the averaging is weighted by the initial measure $\mu_0$ (one can think of sampling initial conditions at random from the probability distribution given by $\mu_0$). By (\ref{DefOccMeasure}) and integration by parts, we get the following relation for all $v \in \C^1(\R^n)$, which we refer to as the (continuous-time) Liouville's equation
\begin{align}\label{EqnLiouville}
	\int\limits_X \nabla v \cdot f\;d\mu & = \int\limits_X \int\limits_0^\infty e^{-\beta t} \nabla v(\varphi_t(x)) f(\varphi_t(x)) \; dt \; d\mu_0(x) = \int\limits_X \int\limits_0^\infty e^{-\beta t} \frac{\partial}{\partial t} v(\varphi_t(x)) \; dt \; d\mu_0(x)\notag\\
	& = -\int\limits_X v \; d\mu_0 + \beta \int\limits_X \int\limits_0^\infty e^{-\beta t}v(\varphi_t(x)) \; d\mu_0(x) = \beta\int\limits_X v \; d\mu - \int\limits_X v \; d\mu_0.
\end{align}
The above equation is also a direct application of basic semigroup theory. For $v \in \C^1(\R^n)$, $Av:= \nabla v \cdot f$ is the action of the generator of the Koopman semigroup $T_tv := v \circ \varphi_t$ whose adjoint is called Perron-Frobenius semigroup and is given by the push forward $P_t\mu (C) := \mu(\varphi_t^{-1}(C))$ for measureable sets $C \subset X$. Hence the Laplace transform $\int\limits_0^\infty e^{-\beta t} P_t\mu_0 \; dt$ gives the resolvent $(\mathrm{Id} - \beta P_t)^{-1} \mu_0$ (\cite{c2neu} Theorem 1.10).

Heuristically the constraint that $\mu$ is a measure on $X$ (and not on a larger set) enforces that $\mu_0$ cannot be supported on a larger set than the MPI set $M_+$; otherwise, those points will flow out of $X$ in finite time, thereby making the support of $\mu$ larger than $X$.

A formal proof of the fact that whenever $(\mu,\mu_0)$ satisfies (\ref{EqnLiouville}), then $\supp(\mu_0)$ is contained in $M_+$ is given in \cite{c10neu}. On the other hand we can always choose $\mu_0$ to be the restriction of the Lebesgue measure to $M_+$ and $\mu$ its corresponding occupation measure (\ref{DefOccMeasure}). Hence maximizing the support of a measure $\mu_0$ satisfying the Liouville equation gives $M_+$. However, direct maximization of the support is computationally challenging. In order to circumvent this challenge, we follow the same strategy as in~\cite{c10neu}. Instead of maximizing the support we will maximize the mass $\mu_0(X)$ under the condition that $\mu_0$ is dominated by the Lebesgue measure which is equivalent to $\mu_0 + \hat{\mu}_0 = \lambda\big|_X$ for a $\hat{\mu}_0 \in M(X)$. That gives the following linear program
\begin{equation}\label{OptimizationProbMPI}
	\begin{tabular}{llc}
		$\sup$ & $\mu_0(X)$&\\
		s.t. & $ \mu_0,\hat{\mu}_0,\mu \in M(X)$ &\\
			 & $ \int\limits_X \beta v - \nabla v \cdot f \; d\mu = \int\limits_X v \; d\mu_0$ & $\forall v \in \C^1(\R^n)$\\
			 & $ \mu_0 + \hat{\mu}_0 \hspace{1mm} = \lambda\big|_X$ &
	\end{tabular}
\end{equation}
The invariance conditions for the reversed time direction is imposed by adding the Liouville equation (\ref{EqnLiouville}) induced by the vector field $-f$. This yields the following LP
\begin{align}\label{OptimizationProbA1}
	\begin{tabular}{llc}
		$\sup$ & $\mu_0(X)$\\
		s.t. & $ \mu_0,\hat{\mu}_0,\mu_+,\mu_- \in M(X)$ &\\
			 & $ \int\limits_X \beta v^1 - \nabla v^1 \cdot f \; d\mu_+ = \int\limits_X v^1 \; d\mu_0$ & $\forall v^1 \in \C^1(\R^n)$\\
			 & $ \int\limits_X \beta v^2 + \nabla v^2 \cdot f\; d\mu_- = \int\limits_X v^2 \; d\mu_0$ & $\forall v^2 \in \C^1(\R^n)$\\
			 & $ \mu_0 + \hat{\mu}_0 \hspace{3mm} = \lambda\big|_X$ &
	\end{tabular}
\end{align}
We have the following proposition.
 
\begin{proposition}\label{PropLowerBound}
	$\lambda(\A)$ is a lower bound for (\ref{OptimizationProbA1}).
\end{proposition}
\begin{proof} For a measureable set $C$ let $\mu_0(C) := \lambda(\A \cap C)$ and $\hat{\mu}_0(C):= \lambda\left(\left(X \setminus \A\right) \cap C\right)$. Then $\mu_0, \hat{\mu}_0 \in M(X)$ and $\mu_0 + \hat{\mu}_0 = \lambda\big|_X$. Let $\mu_+$ and $\mu_-$ be the corresponding occupation measures with discount factor $\beta >0$ defined by (\ref{DefOccMeasure}) forward and backward in time. So $(\mu_0,\hat{\mu}_0,\mu_+,\mu_-)$ is feasible and has objective value $\mu_0(X) = \lambda(\A \cap X) = \lambda(\A) = \lambda\big|_X(\A)$.
\end{proof}

Unfortunately, for the primal approach, it is typical that the global attractor has vanishing Lebesgue volume. This is the case for globally asymptotically stable fixed points or periodic orbits but also for more complex dynamical systems as the Lorenz system. In such cases the trivial measure $\mu = 0$ is a minimizer of the primal LP. Although the primal LP gives the correct volume, it is not of much (direct) help of finding the global attractor. But the dual LP gives more insight into the global attractor as we will see in the next section.

\section{The dual LP}
The dual LP of (\ref{OptimizationProbA1}) is given by
\begin{equation}\label{OptimizationProbAp0}
	\begin{tabular}{llc}
		$\inf$ & $\int\limits_X w \; d\lambda$&\\
		s.t. & $(w,v_1,v_2)\in \C(\R^n) \times \C^1(\R^n) \times \C^1(\R^n)$&\\
			 & $-v^1-v^2 + w \geq \one$& \\
			 & $w \geq 0$ &\\
			 & $ \beta v^1 - \nabla v^1 \cdot f \geq 0$ &\\
			 & $ \beta v^2 + \nabla v^2 \cdot f \geq 0$ & 
	\end{tabular}
\end{equation}

The following lemma is a direct generalization of a corresponding result from~\cite{c10neu}. We state the proof taken from ~\cite{c10neu} as well because it is crucial for the paper.
 
\begin{lemma}\label{LemmaPositivity}
	We have $v^1 \geq 0$ on $M_+$, $v^2 \geq 0$ on $M_-$ and hence $w \geq 1$ on $\A$.
\end{lemma}

\begin{proof} It suffices to show $v^1(x) \geq 0$ on $M_+$ because then the same arguments apply to $v^2$ (in reversed time direction). The fact that $w \geq 1$ on $\A$ follows then from $-v^1-v^2+w\geq 1$ and $\A = M_+ \cap M_-$.\\
Let us proof $v^1\geq 0$ on $M_+$. Feasibility implies
\begin{equation*}
	\beta v^1(x) \geq \nabla v^1(x) f(x) = \frac{\mathrm{d}}{\mathrm{ds}} v^1(\varphi_s(x)) \bigg|_{s = 0}.
\end{equation*}
Hence we get for all $x \in M_+$ and $t \in \R_+$
\begin{equation*}
	\beta v^1(\varphi_t(x)) \geq \frac{\mathrm{d}}{\mathrm{ds}} v^1(\varphi_s(\varphi_t(x))\bigg|_{s = 0} = \frac{\mathrm{d}}{\mathrm{ds}} v^1(\varphi_{s+t}(x))\bigg|_{s = 0} = \frac{\mathrm{d}}{\mathrm{dt}} v^1(\varphi_{t}(x))
\end{equation*}
Gronwall's lemma gives $v^1(\varphi_t(x)) \leq e^{\beta t} v^1(x)$ for all $t \in \R_+$, i.e.
\begin{equation}\label{EqnV^1Positive}
	v^1(x) \geq e^{-\beta t} v^1(\varphi_t(x)).
\end{equation}
But since $\varphi_t(x) \in X$ for all $t \in \R_+$ because $x \in M_+$ the right-hand side in (\ref{EqnV^1Positive}) converges to $0$ as $t \rightarrow \infty$. Hence $v^1(x) \geq 0$ for all $x \in M_+$.
\end{proof}

The above lemma indicates how a solution of the dual LP should look like, i.e., that $w = \mathbb{I}_\A$. Since $\mathbb{I}_\A$ is not continuous on $X$ (if $\A \neq X)$), this implies that the solution to~(\ref{OptimizationProbAp0}) is not attained; however, a minimizing sequence exists and its infimum is equal to the supremum in the primal problem~(\ref{OptimizationProbA1}), i.e., there is not duality gap. This is formalized in the following crucial result.

\begin{theorem}\label{Thm}
	For all $\beta > 0$ there is no duality gap and the optimal value of ~(\ref{OptimizationProbA1}) and (\ref{OptimizationProbAp0}) is given by $\lambda(\A)$. The infimum in the dual program is not attained unless $\A = X$. For a feasible solution $(v^1,v^2,w)$ of the dual problem we have $\A \subset w^{-1}([1,\infty))$.
\end{theorem}
 
This statement shows that the global attractor can be approximated by superlevel sets of functions $w$ obtained from the dual problem and as the feasible solutions approach the optimum this approximation gets tight.\\
Before we prove the theorem we want to mention the following. Since the global attractor is given by the intersection of $M_+$ and $M_-$ set we can first apply the linear program from \cite{c10neu} to find $M_+$ and in the next step apply the same problem to the dynamical system with reversed time direction to find the part of $M_-$ laying in $M_+$. Since it was shown in \cite{c10neu} that there is no duality gap in each step there will be no duality gap for the LP given here. The arguments used in \cite{c10neu} use infinite-dimensional LP theory while we will give a constructive proof.\\
In the case of regularizing discount factor $\beta > \mathrm{Lip}(f)$ we will construct a sequence of feasible solutions $(v^1_m,v^2_m,w_m)_{m \in \N}$ such that $w_m \rightarrow \mathbb{I}_\A$ in $\mathrm{L}^1(X,\lambda)$; in particular this is a minimizing sequence by Lemma~\ref{PropLowerBound}.\\
Before proving Theorem~\ref{Thm},  let us state the following classical result 
\cite[Theorem 2.29]{c17neu}.
 
\begin{proposition}\label{PropSmoothDistance}
	For each closed set $C \subset \R^n$ there exists a bounded function $p\in \C^\infty(\R^n)$ such that $p^{-1}(\{0\}) = C$ and $p(x) \geq 0$ for all $x \in \R^n$.
\end{proposition}
 

\begin{proof} \textit{of Theorem \ref{Thm}.} We start with the easy part, namely that the superlevel sets $w^{-1}([1,\infty))$ give outer approximations of the global attractor $\A$. By Lemma \ref{LemmaPositivity} we have $w(x) \geq 1$ on $\A$. This also shows that the infimum can only be attained if $\mathbb{I}_\A$ is continuous on $X$ which is the case if and only if $\A$ is one connected component of $X$. But since $\A$ has non-empty intersection with every connected component of $X$ (see \cite{c28neu}) it follows that $\A = X$. And in that case we can choose $w = 1$ everywhere and $v^1 = v^2 = 0$.\\
For the rest, we only cover the case $\beta > \mathrm{Lip}(f)$. We need this technical assumption in order to guarantee that our construction gives a sufficient regular function (namely $\C^1$). Note that the general case is covered by applying the arguments from \cite{c10neu} twice, once forward in time and once in reversed time direction.\\
The idea is that we will define suitable functions $v^1_m,v^2_m$ that satisfy the equation $\beta v_m^1 - \nabla v^1_m f \geq 0$ and $\beta v^2 + \nabla v^2 \cdot f\geq 0$ respectively and build a minimizing sequence based on those functions.\\
Let us start with a technical construction motivated by a similar construction in~\cite{c12neu}. We want to recognize which points leave $X$. Without loss of generality we can assume $f$ being globally Lipschitz with globally  Lipschitz derivative since, under our assumptions, one can always modify $f$ outside of $X$ in this fashion. This does not affect the result because the global attractor is only determined by the dynamics in $X$. We denote the corresponding flow also by $\varphi$. We note that this flow exists for all $(t,x)\in\R^{n+1}$ and we have $x \notin M_+$ if and only if $\varphi_t(x) \notin X$ for some $t \in \R_+$. Let us choose $p \in \C^\infty(\R^n)$ bounded such that $p^{-1}(\{0\}) = X$ and $p>0$ everywhere else. For $x \in \R^n$ define
\vglue-5mm
\begin{equation}\label{EqnDefVn}
	v^1(x):= - \int\limits_0^\infty e^{-\beta t} p(\varphi_t(x))\; dt.
\end{equation}
Then $v^1(x) < 0$ if and only if there exists a time $t \in \R_+$ for which we have $\varphi_t(x) \notin X$. In particular $v^1 < 0$ on $X \setminus M_+$. We will see that $v^1$ satisfies $\beta v^1 - \nabla v^1 \cdot f \geq 0$ on $X$. We have $\|\partial_{x}\varphi_t(x)\| \leq \overline{M} e^{\mathrm{Lip}(f)t}$ for some $\overline{M} >0$ and all $t \in \R_+$. Thanks to $\beta > \mathrm{Lip}(f)$ we can interchange integration and differentiation and get for all $x \in \R^n$
\begin{equation*}
	Dv^1(x) = -\int\limits_0^\infty e^{-\beta t} \partial_x\left( p(\varphi_t(x))\right) \; dt =  - \int\limits_0^\infty e^{-\beta t} Dp(\varphi_t(x)) \partial_x \varphi_t(x) \; dt.
\end{equation*}
Further
\begin{eqnarray*}
	\beta v^1 (x) & = & -\beta \int\limits_0^\infty e^{-\beta t} p(\varphi_t(x)) \; dt \overset{p.i.}{=} -p(x) + \int\limits_0^\infty e^{-\beta t} Dp(\varphi_t(x)) f(\overline{\varphi}_t(x)) \; dt\\
		 & = & -p(x) - \int\limits_0^\infty e^{-\beta t} Dp(\varphi_t(x)) \partial_{x} \varphi_t(x) f(x) \; dt = -p(x) + Dv^1(x) f(x).
\end{eqnarray*}
where we have used in the third line the following relation
\begin{equation*}
	\partial_{x_0} \varphi_t(x_0) \cdot f(x_0) = \partial_t \varphi_t(x_0) = f(\varphi_t(x_0)).
\end{equation*}
This relation is the only part where we need $f$ to have a locally Lipschitz continuous derivative.\\
Since $p$ is vanishing on $X$ we have $\beta v^1 - \nabla v^1 \cdot f = 0$ on $X$ and $v^1 (x) < 0$ for $x \notin M_+$. Proceeding similarly backward in time we find $v^2 \in \C^1(\R^n)$ that satisfies $\beta v^2 + \nabla v^2 \cdot f \geq 0$ on $X$ and $v^2 (x) < 0$ for $x \notin M_-$. In particular the triple $(v^1_m,v^2_m,w_m):=(m\cdot v^1,m \cdot v^2,\max\{0,1+m\cdot v^1+m\cdot v^2\})$ is feasible and as $m \rightarrow \infty$ we have $w_m \searrow \mathbb{I}_\A$. It follows from the monotone convergence theorem that $\int\limits_X w_m \; d\lambda \rightarrow \int\limits_X \mathbb{I}_\A \; d\lambda = \lambda(\A)$. By Proposition \ref{PropLowerBound} we know that $\lambda(\A)$ is a lower bound for the primal problem while the above shows that $\lambda(\A)$ is an upper bound of the dual problem. Weak duality gives that $\lambda(\A)$ is the optimal value for both the primal and dual LP.
\end{proof}

Note that for $\beta > \mathrm{Lip}(f)$ we have constructed a feasible solution $(v^1,v^2,w)$ such that $w^{-1}([1,\infty) ) = w^{-1}(\{1\}) = \A$.

\section{Discrete time}
In this section, we consider discrete-time systems of the form
\[
x^+ = f(x).
\]

The main difference between continuous and discrete time is that continuous time systems that are induced by ordinary differential equations with locally Lipschitz right-hand side enjoy unique solutions forward and backward in time, which corresponds to injectivity of the flow functions $\varphi_t$. Discrete-time systems do not have this property in general; these systems are also well defined for maps $f:X \rightarrow \R^n$ that are not injective. The problem with our approach occurs when we want to invert time; for non-injective systems, it means that we may have multiple predecessors. Before we state the analog result of Theorem \ref{Thm} we will define the MNI set $M_-$ for discrete time systems and proceed analogously as in the continuous time. The definition of the MPI set $M_+$ does not change because the problem of multiple successors does not occur. We will give an analog LP that gives outer bounds for the global attractor for discrete systems. If $f$ is injective those outer bounds get sharp. From now on we will always refer to a dynamical system of the form $(X,(f^m)_{m \in \N})$ for a compact subset $X \subset \R^n$ and a continuous map $f:\R^n \rightarrow \R^n$, where $f^{0} := \mathrm{Id}$ and $f^{m+1} := f \circ f^m$ for $m \in \N$.
 
\begin{definition}[Global attractor]
	A global attractor for the discrete time dynamical system $(X,(f^m)_{m \in \N})$ is a minimal compact set $\A$ such that
	\[
	 \lim_{m\to\infty} \dist(f^m(M_+),\A)  = 0.
	 \]
\end{definition}
 
As for continuous time dynamical systems global attractors are unique,  $f(\A) = \A$ and they enjoy many interesting properties. We refer to \cite{c28neu} for such properties and the following result.
 
\begin{theorem}[Existence of the global attractor]\label{ThmExistenceGlobAttractor}
	If $X$ is compact the global attractor exists.
\end{theorem}
 
\begin{definition} For a discrete time dynamical system $(X,(f^m)_{m \in \N})$ we define the maximal negatively invariant set as	$M_-:= \{ x \in X: \forall m \in \N \; \exists x_m  \in X \text{ with } f^m(x_m) = x\}$.
\end{definition}
 
\begin{remark} $M_-$ is given by $\bigcap\limits_{m \in \N} f^m(X)$ and hence is closed since $f$ is continuous and $X$ compact.
\end{remark}
Next, we show that the global attractor in discrete time is also given by all the points that stay in $X$ for all positive and negative times.
 
\begin{proposition}\label{PropGlobAttractorMaxInvSetDiscrete}
	The global attractor for the discrete time dynamical system $(X,(f^m)_{m \in \N})$ is given by $M_+ \cap M_-$.
\end{proposition}
 
\begin{proof} By Theorem \ref{ThmExistenceGlobAttractor} the global attractor exists. Let $\A$ be the global attractor. The condition $f(\A) = \A$ implies that $\A \subset M_+ \cap M_-$. On the other hand the set $M_+ \cap M_-$ satisfies $f(M_+ \cap M_-) = M_+ \cap M_-$. This can checked as follows. Let $x \in M_- \cap M_+$. Then the orbit of $f(x)$ is contained in the orbit of $x$; hence $f(x) \in M_+$ and for a sequence of points $x_m$ such that $f^m(x_m) = x$ we have $f^m(x_{m-1}) = f(x)$ for all $m \in \N$ with $x_0 := x$, hence also $f(x) \in M_-$. On the other hand obviously $x = f(x_1)$. Which shows $f(M_+ \cap M_-) = M_+ \cap M_-$. Hence $\A \supset M_+ \cap M_-$ by definition of the global attractor.
\end{proof}

The occupation measure with discount factor $\alpha \in [0,1)$ is defined by
\begin{equation}\label{DefOccMeasureDiscrete}
	\mu(C) := \int\limits_X\sum\limits_{m = 0}^\infty \alpha^m \mathbb{I}_{C}(f^m(x)) \; d\mu_0(x),
\end{equation}
where the discount factor $\alpha \in [0,1)$ plays the role of $e^{-\beta}$ in continuous time. And in an analogous way to the continuous time case we get a primal LP for the MPI set $M_+$ (see \cite{c10neu} problem (6))
\begin{align*}
		\sup \; \; & \mu_0(X) = \sup \; \; \int\limits_X \one \; d\mu_0\\
		\text{s.t.}\; \;  & \mu_0,\hat{\mu}_0,\mu \in M(X)\\
			 & \int\limits_X v - \alpha v \circ f \; d\mu = \int\limits_X v \; d\mu_0 & \forall v \in \C(\R^n) \notag \\
			 & \mu_0 + \hat{\mu}_0 = \lambda\big|_X.
\end{align*}
In the continuous time case, we reversed time by making the vector field point in the reverse direction. In discrete time we reverse time by considering the map $f^{-1}$ but this might be multivalued. In order to reverse the time direction we switch the position of $f$ in the discrete Liouville equation
\begin{align}\label{OptimizationDiscreteProbApPrimal}
		p^*:= \sup \; \; & \mu_0(X) = \sup \; \; \int\limits_X \one \; d\mu_0 \notag \\
		\text{s.t.}\; \;  & \mu_0,\hat{\mu}_0,\mu_+,\mu_- \in M(X) \notag \\
			 & \int\limits_X v - \alpha v \circ f \; d\mu_+ = \int\limits_X v \; d\mu_0 & \forall v \in \C(\R^n) \notag \\
			 & \int\limits_X v \circ f - \alpha v \; d\mu_- = \int\limits_X v \; d\mu_0 & \forall v \in \C(\R^n) \notag \\
			 & \mu_0 + \hat{\mu}_0 = \lambda\big|_X.
\end{align}

This gets clearer in the dual program that is given by
\vglue-10mm
\begin{align}\label{OptimizationDiscreteProbAp0}
		d^*:= \inf \; \; & \int\limits w \; d\lambda\\
		\text{s.t.}\; \;  & v^1,v^2w \in \C(\R^n)\notag\\
			 & -v^1-v^2 + w \geq \one\notag\\
			 & w \geq 0\label{DiscFeasWPositive}\\
			 & v^1 - \alpha v^1 \circ f \geq 0\label{DiscreteLiouvilleInequality}\\
			 & v^2 \circ f - \alpha v^2 \geq 0 \label{DiscLiouvSwitchedF}
\end{align}
If $f$ is invertible on $X$ then the reversed time equation would read $v^2(y) - \alpha v^2(f^{-1}(y)) \geq 0$ and the substitution $x = f^{-1}(y)$, i.e. $y = f(x)$, gives $v^2(f(x)) - \alpha v^2(x) \geq 0$. But this is exactly the term in the last line.\\
Again we start with stating a lower bound.
 
\begin{lemma}\label{LemmaLowerBound}
	Each feasible solution $(v^1,v^2,w)$ satisfies $v^1 \geq 0$ on $M_+$ and $v^2 \geq 0$ on $M_-$ and $w \geq \one$ on $\A$. In particular $w^{-1}([1,\infty)) \supset \A$ and $d^* \geq \lambda(\A)$.
\end{lemma}
 
\begin{proof} For feasible $(v^1,v^2,w)$ we have for each $x \in M_+$ that $v^1(x) \geq \alpha v^1(f(x)) \geq \ldots \geq \alpha^m v^1(f^m(x)) \rightarrow 0$ as $m \rightarrow \infty$ and for each $x \in M_-$ with $x_m \in X$ such that $f^m(x_m) = x$ that $v^2(x) = v^2(f^m(x_m)) \geq \alpha^m v^2(x_m) \rightarrow 0$ as $m \rightarrow \infty$. Hence for $w$ we get $w \geq \one+v^1+v^2 \geq \one$ on $\A = M_+ \cap M_-$ and $\int\limits_X w \; d\lambda \geq \lambda(\A)$ since $w \geq 0$ everywhere on $X$ by (\ref{DiscFeasWPositive}).
\end{proof}

Next, we solve the discrete dual Liouville equation with a right-hand side.
 
\begin{lemma}\label{LemmaDiscreteLiouville}
	Let $\alpha \in [0,1)$. Let $p\in \C(R^n)$ be bounded and $F:\R^n \rightarrow \R^n$ be continuous
	. Then
	\begin{equation}\label{EqnConstructV1}
		v:\R^n \rightarrow \R, \;  v(x) := \sum\limits_{k = 0}^\infty \alpha^k p(F^k(x))
	\end{equation}
	is continuous and solves
	\begin{equation}\label{EqnDiscLiouville}
		v - \alpha v\circ F = p.
	\end{equation}
\end{lemma}
 
The proof is just the discrete analog of the calculations from the proof of Theorem \ref{Thm}.\\
Now we have everything we need to state the main theorem in the discrete case. But we need an additional technical condition on $f$ in order to reverse time direction.
 
\begin{theorem}\label{TheoremLPOptimalValueDiscrete}
	Let $X \subset \R^n$ be compact and $f:R^n \rightarrow \R^n$ be continuous and injective on $X$. Then for all $\alpha \in (0,1)$ the optimal values $p^*$ and $d^*$ of the above LPs are given by $\lambda(\A)$. For each feasible solution $(v^1,v^2,w)$ of the dual problem we have $\A \subset w^{-1}([1,\infty))$ and $\lambda(w_m^{-1}([1,\infty)) \rightarrow \lambda(\A)$ for any minimizing sequence $(v^1_m,v^2_m,w_m)$ of feasible solutions.
\end{theorem}
 
\begin{proof} The idea is again to construct a minimizing sequence for the dual problem, with objective value converging to $\lambda(\A)$, explicitly. This, together with Lemma \ref{LemmaLowerBound} and an explicit construction of $\lambda(\A)$ as a lower bound for the primal LP, guarantees that the optimal value of both LPs is given by $\lambda(\A)$.\\
We choose any bounded function $p\in \C(\R^n)$ such that $p^{-1}(\{0\}) = X$ and $p(x) < 0$ for $x \notin X$ (Proposition \ref{PropSmoothDistance}) and apply Lemma \ref{LemmaDiscreteLiouville} to find a function $v^1$ that satisfies $v^1 - \alpha v^1 \circ f = p$. The function $v^1$ solves the discrete Liouville inequality (\ref{DiscreteLiouvilleInequality}) since $p = 0$ on $X$. And further by construction of $v^1$ in (\ref{EqnConstructV1}) we have $v^1(x) < 0$ if and only if $x \notin M_+$. We want to do something similar for the time reversed equation. Therefore let us start with constructing a ``time reversed" system. Since $X$ is compact and $f$ is continuous and injective on $X$ the restriction $f\big|_X$ is a homeomorphism onto its image, i.e. $f\big|_X:X \rightarrow f(X)$ is continuously invertible. Let $\tilde{g}:f(X) \rightarrow X$ be its inverse. By Tietze's extension theorem let $g:\R^n \rightarrow \R^n$ be a continuous extension of $\tilde{g}$.
Let now $q\in \C(\R^n)$ be any bounded continuous function. Again applying Lemma \ref{LemmaDiscreteLiouville} we get a function $v^2$ that solves
\begin{equation}\label{EqnV2}
	v^2(y) - \alpha v^2(g(y))= q(y) \; \;	\text{ for all } y \in \R^n.
\end{equation}
If we plug in $y = f(x)$ to (\ref{EqnV2}) for $x \in X$ we get
\begin{align*}
	q(f(x)) = q(y) = v^2(y) - \alpha v^2(g(y)) = v^2(f(x)) - \alpha v^2(g(f(x)) = v^2(f(x)) -\alpha v^2(x).
\end{align*}
This means $v^2$ satisfies (\ref{DiscLiouvSwitchedF}) for the dual problem if $q(f(x)) \geq 0$ for all $x \in X$. Hence let us take a bounded function $q \in \C(\R^n)$ such that $q \leq 0$ and $q^{-1}(\{0\}) = f(X)$. Let us check when $v^2$ does not vanish. By construction of $v^2$ from Lemma \ref{LemmaDiscreteLiouville} we have $v^2(x) = 0$ if and only if $g^m(x) \in q^{-1}(\{0\}) = f(X)$ for all $m \in \N$. This means $x \in f(X)$ and hence $X \ni f\big|_X^{-1}(x) = g(x) \in f(X)$. By induction it follows $g^m(x) \in X$ for all $m \in \N$, i.e. $x \in M_-$, and $x \in f(X)$. It looks like the additional condition $x \in f(X)$ is restrictive, but it is not since we are interested in the global attractor $\A$ that satisfies $\A = f(\A) \subset f(X)$. Let us construct a minimizing sequence of feasible solutions to the dual LP. As in the continuous time case let
\begin{equation*}
	(v^1_m,v^2_m,w_m) := (m\cdot v^1,m\cdot v^2,\max\{0,\one+m\cdot v^1 + m\cdot v^2\}.
\end{equation*}
Since $v^1$ and $v^2$ are solutions of the discrete dual Liouville equation the triple $(v^1_m,v^2_m,w_m)$ is feasible. Since $v^1$ is vanishing only on $M_+$ and $v^2$ on $f(X) \cap M_-$ and both functions are negative everywhere else we get $w_m(x) \leq 1$ and $w_m (x) = 1$ exactly for $f \in M_+ \cap f(X) \cap M_- = \A$ by Proposition \ref{PropGlobAttractorMaxInvSetDiscrete} and $w_m \searrow \mathbb{I}_\A$ pointwise as $m \rightarrow \infty$. Hence by the monotone convergence theorem, it follows $\int\limits_X w_m \; d\lambda \rightarrow \lambda(\A)$. By Lemma \ref{LemmaLowerBound} it follows for all $m \in \N$ that $(v^1_m,v^2_m,w_m)$ is a minimizing sequence. Let now $(v^1_m,v^2_m,w_m)$ be any minimizing sequence. By Lemma \ref{LemmaLowerBound} we get $\lambda(\A) \leq \lambda(w_m^{-1}([1,\infty)) \leq \int\limits_X w_m \; d\lambda \rightarrow \lambda(\A)$ as $m \rightarrow \infty$.\\
To check that $\lambda(\A)$ is also the optimal value of the primal LP we proceed as in the continuous time case by showing that $\lambda(\A)$ is a lower bound for the primal LP. Then it follows that the optimal value is given by $\lambda(\A)$ by weak duality, i.e. from
\begin{equation*}
	\lambda(\A) \leq p^* \leq d^* = \lambda(\A).
\end{equation*}
To see $p^* \geq \lambda(\A)$ we will find a feasible point $(\mu_0,\hat{\mu}_0,\mu_+,\mu_-) = (\lambda\big|_\A,\lambda\big|_{X \setminus \A},\mu_+,\mu_-)$. Then the objective value is $\mu_0(X) = \lambda(A)$. Let $\mu_+$ be defined by the right-hand side of (\ref{DefOccMeasureDiscrete}) with $\mu_0 = \lambda\big|_\A$. In order to find $\mu_-$, we first let $\nu$ be the defined by the right-hand side of~(\ref{DefOccMeasureDiscrete}) with $\mu_0 = \lambda\big|_\A$ and $f$ replaced by $g$.
This relates to reversing time direction. Note that from $f(\A) = \A$ it follows that also $g(\A) = \A$ and hence that $\supp(\nu) = \A$. Set $\mu_{-}(C) := \nu(g^{-1}(C))$, i.e., $\mu_{-}$ is the pushforward measure of $\nu$ by $g$. Since measures defined by~(\ref{DefOccMeasureDiscrete}) solve the discrete Liouville equation (first constraint of~(\ref{OptimizationDiscreteProbApPrimal})) and $g\big|_\A = f\big|_\A^{-1}$ we get for all $v \in \C(\R^n)$
\begin{eqnarray*}
	\int\limits_X v \; d\lambda\big|_\A & = & \int\limits_X v - \alpha v \circ g \; d\nu = \int\limits_{\supp(\nu)} v - \alpha v \circ g \; d\nu = \int\limits_\A v \circ f \circ g - \alpha v \circ g\; d\nu\\
										& = & \int\limits_{g(\A)} v \circ f - \alpha v \; d\mu_{-} = \int\limits_X v \circ f - \alpha v \; d\mu_{-},
\end{eqnarray*}
where we used the fact that $\supp(\mu_-) = g(\A) = \A \subset X$. Therefore, $\mu_-$ solves the reversed time Liouville equation (second constraint of~(\ref{OptimizationDiscreteProbApPrimal})). It follows that $(\lambda\big|_\A,\lambda\big|_{X \setminus \A},\mu_+,\mu_-)$ is feasible with objective value $\lambda(\A)$ which is what remained to be shown.
\end{proof}

The infimum in the dual program is not attained unless $\A$ is a union of connected components of $X$, because only then $\mathbb{I}_\A$ is continuous. But we have seen in the proof that there exist feasible solutions $(v^1,v^2,w)$ such that $\A = w^{-1}([1,\infty))$ even before their objective value approaches the optimal value.

\section{Solving the linear programs}
If we assume algebraic structure of the problem this can be exploited to solve the infinite dimensional primal LP by a hierarchy of finite dimensional semidefinite programs whose optimal values converge to the solution of the infinite dimensional LP. The resulting SDPs are relaxations of the original LP since they describe truncated versions of the moment problem. The SDPs can be solved by freely available software. Similarly, the dual LP tightens to a sum-of-squares problem, which also leads to a hierarchy of SDPs. This is a standard procedure and we refer to~\cite{c16neu} for details.

\textbf{Assumption:} The vector field $f$ is polynomial and $X$ is a compact basic semi-algebraic set, that is, there exist polynomials $p_1,\ldots,p_j \in \R[x_1,\ldots,x_n]$ such that $X = \{x \in \R^n: p_i(x) \geq 0 \text{ for } i = 1,\ldots,j\}$. Further we assume that one of the $p_i$ is given by $p_i(x) = R_X^2 - \|x\|_2^2$ for some large enough $R_X \in \R$.

If there is no such $p_i$ then by compactness of $X$ we can add the redundant inequality $p_{k+1}:= R_X - \|x\|_2^2 \geq 0$ for the smallest radius $R_X$ such that $\overline{B_{R_X}(0)}$ contains $X$. This will be useful in order to apply Putinar's Positivstellensatz (see~\cite{c27neu}). We will only state the dual tightenings of the problems because these provide guaranteed outer approximations of the global attractor, while this may not be true for the primal problem. In order to solve the infinite dimensional problem we first replace the space of continuous functions by the space of polynomials; this is justified by the Stone-Weierstra{\ss} theorem. Then we truncate the degree of the polynomials to get tightenings of the dual problem in form of finite dimensional SDPs. The idea is to replace all variables, i.e. functions, in the LP by elements of $\R[x]_k$, polynomials of degree at most $k$, and apply Putinar's Positivstellensatz to reformulate positivity as a sum-of-squares constraint. The corresponding tightenings truncated at degree $k$ for the problem in continuous time read as
\begin{equation}\label{OptimizationSDPDegreedContinuous}
	\begin{tabular}{llc}
		$d_k:=$ & $\inf\limits_{v^1,v^2,w,\{q_i\},\{t_i\},\{r_i\},\{s_i\} }\;\;\; \mathbf{w}' \mathbf{l}$& \vspace{0.5mm}\\
		s.t.
			 & $-v^1-v^2 + w - 1 = q_0 + \sum\limits_{i = 1}^j q_i p_i$& \\
			 & $w(x) = t_0 + \sum\limits_{i = 1}^j t_i p_i$&\\
			 & $ \beta v^1 - \nabla v^1 \cdot f = r_0 + \sum\limits_{i = 1}^j r_i p_i$ &\\
			 & $ \beta v^2 + \nabla v^2 \cdot f = s_0 + \sum\limits_{i = 1}^j s_i p_i$ & 
	\end{tabular}
\end{equation}
and for discrete time systems respectively
\begin{equation}\label{OptimizationSDPDegreedDiscrete}
	\begin{tabular}{llc}
		$d_k:=$ & $\inf\limits_{v^1,v^2,w,\{q_i\},\{t_i\},\{r_i\},\{s_i\} }\;\;\; \mathbf{w}' \mathbf{l} $& \vspace{0.5mm}\\
		s.t. & $v^1,v^2,w \in \R[x]_k$ &\\
			 &  $w-v^1-v^2-1= q_0 + \sum\limits_{i = 1}^j q_i p_i$ &\\
			 &  $w = t_0 + \sum\limits_{i = 1}^j t_i p_i$&\\
			 &  $v^1 - \alpha v^1\circ f  = r_0 + \sum\limits_{i = 1}^j r_i p_i$&\\
			 &  $v^2\circ f - \alpha v^2 = s_0 + \sum\limits_{i = 1}^j s_i p_i$,&
	\end{tabular}
\end{equation}
where $\mathbf{w}'$ is the vector of coefficients of the polynomial $w$ and $\mathbf{l}$ is the vector of the moments of the Lebesgue measure over $X$ (i.e., $\mathbf{l}_\alpha = \int_X x^\alpha \, d\lambda(x)$, $\alpha \in \mathbb{N}^n$, $\sum_i \alpha_i \le k$), both indexed in the same basis of $\R[x]_k$; hence $\mathbf{w}'\mathbf{l} = \int\limits_X w(x) \; d\lambda(x)$. The decision variables $v^1, v^2 , w$ are polynomials in $\R[x]_k$ whereas $q_0,\ldots,q_j,r_0,\ldots,r_j,s_0,\ldots,s_j,t_0,\ldots,t_j$ are sums of squares of polynomials with degrees such that $q_0$, $t_0$, $r_0$, $s_0$, $q_ip_i$, $t_ip_i$, $r_ip_i$, $s_ip_i$ are all in $\R[x]_k$ for all $i = 1,\ldots,j$. These sum-of-squares optimization problems translate directly to convex SDPs (see, e.g., \cite{c16neu,c23neu}) with high-level modeling software available (e.g., Yalmip~\cite{c18neu}, Gloptipoly~\cite{c7neu}).
 
\begin{theorem}\label{ThmSDPConvergence} For all $k \in \N$ we have $d_k \geq d_{k+1}$ and $d_k\rightarrow \lambda(\A)$ as $k\rightarrow \infty$.
\end{theorem}
 
\begin{proof} We cover both discrete and continuous time simultaneously because the arguments are the same. The inequality $d_k \geq d_{k+1}$ follows immediately since the set of feasible elements is monotonically increasing with $k$. To prove convergence note first that any triple $(v^1,v^2,w)$ that is feasible for the relaxed problem (\ref{OptimizationSDPDegreedContinuous}) and (\ref{OptimizationSDPDegreedDiscrete}) is feasible for the original dual LPs (\ref{OptimizationProbAp0}) and (\ref{OptimizationDiscreteProbAp0}), hence we have $d_k \geq d^*$ for all $k \in \N$. By Theorem~\ref{Thm} or Theorem~\ref{TheoremLPOptimalValueDiscrete} respectively we have $d^* = \lambda(\A)$. To prove that $\lim\limits_{k \rightarrow \infty} d_k \leq d^* = \lambda(\A)$ let $\varepsilon > 0$ and $(v^1,v^2,w)$ be feasible for the dual LP. Then $(v^1+\varepsilon,v^2+\varepsilon,w+\varepsilon)$ is strictly feasible and by compactness and the Stone-Weierstra{\ss} theorem we can find polynomials $\nu^1,\nu^2,\omega \in \R[x_1,\ldots,x_n]$ such that $\max\{\|v^1- \nu^1\|_\infty,\|\nabla v^1 - \nabla \nu^1\|_\infty\}, \max\{\|v^2-\nu^2\|_\infty, \|\nabla v^2 - \nabla \nu^2\|_\infty\}< \frac{\beta}{1+\beta} \varepsilon$ in the continuous time case and $\|v^1- \nu^1\|_\infty,\|v^2-\nu^2\|_\infty< \frac{1-\alpha}{1+\alpha} \varepsilon$ in the discrete time case and $\|w-\omega\|_\infty < \varepsilon$ in both cases. By the triangle inequality we see that $(\nu^1,\nu^2,\omega)$ is strictly feasible with objective value $\int\limits_X \omega \; d\lambda \leq \int\limits_X w \; d\lambda + \varepsilon \lambda(X)$. Since $\varepsilon > 0$ was arbitrary we see that the optimal value is unchanged when restricting the decision variable to polynomials. The convergence then follows from Putinar's Positivstellensatz \cite{c27neu}.
\end{proof}

\subsection{Converging outer approximations}
In this section, we use solutions to the sum-of-squares program ~(\ref{OptimizationSDPDegreedContinuous}) for continuous time systems and~(\ref{OptimizationSDPDegreedDiscrete}) for discrete time systems to define semialgebraic outer approximations to the global attractor $\A$ and prove their convergence to $\A$. Specifically, we define
\begin{equation}\label{eq:outerAppW}
	Y_k:= \{x \in X: w_k(x) \geq 1\}
\end{equation}
and
\begin{equation}\label{eq:outerApp}
X_k := \{ x\in X \mid \min\{v^1_k(x), v^2_k(x)\} \ge 0  \},
\end{equation}
where ($v^1_k,v^2_k,w_k)$ is a solution to~(\ref{OptimizationSDPDegreedContinuous}) or ~(\ref{OptimizationSDPDegreedDiscrete}) respectively. We will see that $X_k$ provides a better approximation of $\A$ than $Y_k$. We did not use those approximations before because they were not needed in the proofs and we think that the superlevel set $w^{-1}([1,\infty))$ is easier recognized as a reasonable outer approximation of the global attractor closely connected to the optimal value of the dual LP.

\begin{theorem} \label{thm:convOuter}
For each $k\in \N $ we have $Y_k \supset X_k\supset \A$. In addition,
\[
\lim_{k\to \infty} \lambda(Y_k\setminus \A) = \lim_{k\to\infty} \lambda(X_k \setminus \A) = 0.
\]
\end{theorem} 

\begin{proof} For any feasible triple $(v^1,v^2,w)$ (of the dual LP as well as the tightened problems) the condition $\min\{v^1(x),v^2(x)\} \geq 0$ for some $x \in X$, i.e. $v^1(x),v^2(x) \geq 0$, implies, by feasibility, that $w(x) \geq 1+v^1(x)+v^2(x) \geq 1$. This gives $Y_k \supset X_k$ for any $k \in \N$. From Lemma \ref{LemmaPositivity} and Lemma \ref{LemmaLowerBound} respectively it follows for all $x \in \A$ that $v^1(x),v^2(x) \geq 0$, hence also $X_k \supset \A$ for any $k \in \N$. To check convergence it suffices to show only $\lambda(Y_k \setminus \A) \rightarrow 0$ because $\A \subset X_k \subset Y_k$. From $Y_k \supset \A$ we get $\lambda(Y_k \setminus \A) = \lambda(Y_k) - \lambda(\A)$ and it also suffices to check $\lambda(Y_k) \rightarrow \lambda(\A)$. Let $d_k$ be the optimal value of the tightening SDP~(\ref{OptimizationSDPDegreedContinuous}) or~(\ref{OptimizationSDPDegreedDiscrete}) respectively and $(v^1_k,v^2_k,w_k)$ a corresponding minimizer. We have by non-negativity of $w_k$
\begin{eqnarray*}\label{EqnSandwichdkYkA}
	d_k = \int\limits_X w_k \; d\lambda \geq \hspace{-5mm} \int\limits_{w_k^{-1}([1,\infty))} \hspace{-4mm}  w_k \; d\lambda \geq \int\limits_{Y_k} 1 \; d\lambda = \lambda(Y_k) \geq \lambda(\A).
\end{eqnarray*}
By Theorem \ref{ThmSDPConvergence} we have $d_k \rightarrow \lambda(\A)$, and it follows $\lambda(Y_k) \rightarrow \lambda(\A)$.
\end{proof}

The asymptotic convergence to the global attractor was proven for all parameters $\beta > 0$ and $\alpha \in (0,1)$ respectively. But when computing an outer approximation the choice of this parameter has a quantitative effect. Note that the limit case $\beta = 0$ and $\alpha = 0$ respectively corresponds to the problem of finding an invariant measure, while large values of $\beta$ respectively $\alpha$ give high discounting, i.e. the occupation measure takes short time evolution of the system more into account.

\section{Numerical examples}

We present three numerical examples with code available online from
\begin{center}
\footnotesize \url{https://homepages.laas.fr/mkorda/Attractor.zip}
\end{center}

 Two of the examples have strange attractors and one example has a stable limit cycle. The systems with strange attractors are the Lorenz system (continuous time)
\begin{equation}
	\dot{x} = 10(y-x), \;\; \dot{y} = x(28-z)-y, \;\; \dot{z} = xy-\frac{8}{3}z
\end{equation}
and for discrete time, we consider the H\'{e}non map, scaled such that the attractor is inside the unit box, 
\begin{equation}
	x_{m+1} = \frac{2}{3}(1+y_m) - 2.1 x_m^2, \quad  y_{m+1} = 0.45 x_m.
\end{equation}
The third example is the Van--der--Pol oscillator
\begin{equation}
	\dot{x} = 2y,  \quad \dot{y} = -0.8x - 10(x^2-0.21)y.
\end{equation}

The numerical approximations of the attractors were generated by simulation of very long trajectories, discarding the initial portions. The SDP  problems were solved using MOSEK. The figures  Fig. \ref{figLorenz}, Fig. \ref{figHenon} and Fig. \ref{figVanDerPol} show the outer approximations of the global attractors given by $X_k$ from (\ref{eq:outerApp}) arising from the tightening SDPs with degree bound $k = 8$ for the Lorenz system, $k = 8$ and $k = 10$ for the H\'{e}non map and $k = 12$ for the Van-der-Pol oscillator.\\
The outer approximation of the global attractor for the Lorenz system is drawn in light red.
\begin{figure*}[h]
\begin{picture}(400,220)
\put(30,30){\includegraphics[width=55mm]{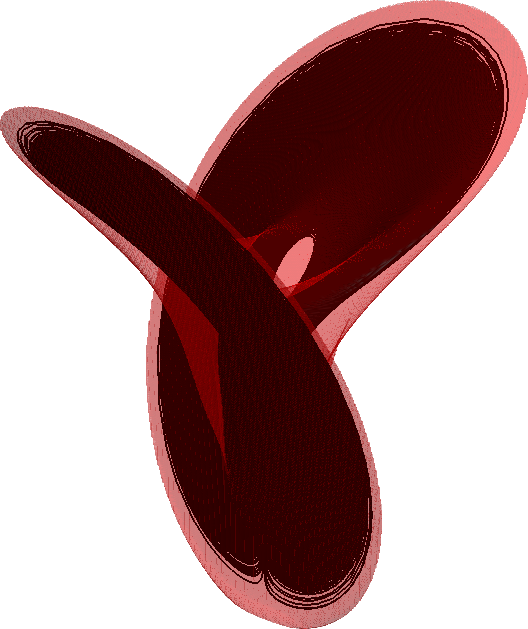}}
\put(260,30){\includegraphics[width=39mm]{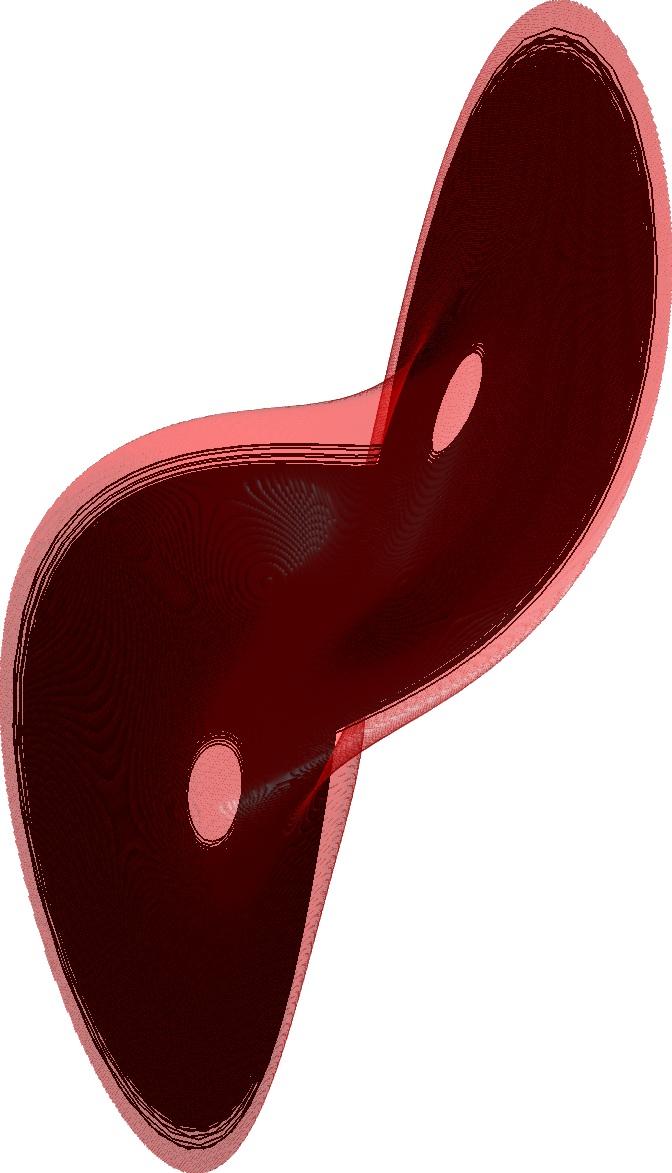}}
\end{picture}
\caption{\footnotesize{Outer approximations for the Lorenz attractor obtained by degree 8 polynomials with discount factor $\beta = 1$ from two angles.}}
\label{figLorenz}
\end{figure*}

The figures in Fig. \ref{figLorenz} show that solutions of the problem truncated at degree 8 are already able to capture the global attractor very well. The time to solve the SDP (ref to dual SOS tightening) was 0.67s with MOSEK 8.1 running on a machine with 4,2 GHz Intel Core i7 and 32 GB 2400 MHz DDR4 RAM.\\

Although we know that the optimal values $d_k = \int w_k$ from the tightening SDPs monotonically decrease to $\lambda(\A)$ it is not guaranteed that the sets $X_k$ and $Y_k$ are monotonically shrinking towards the global attractor. And it is not to be expected that $X_k$ and $Y_k$ show a monotone decay. But that on the other hand allows us to get better results by combining lower degree approximations with higher degree ones.
In addition to that the freedom in the choice of the scalar parameter $\beta$ respectively $\alpha$ allows further refinement. The right pane of Figure~\ref{figHenon} shows the intersection of the outer approximations obtained by the tightening SDPs up for degrees $4,6$ and $8$ and a scalar grid of the parameter $\alpha$. For the H\'{e}non map the outer approximation is given by the grey colored area.
\begin{figure*}[h]
\begin{picture}(300,210)
\put(30,25){\includegraphics[width=62mm]{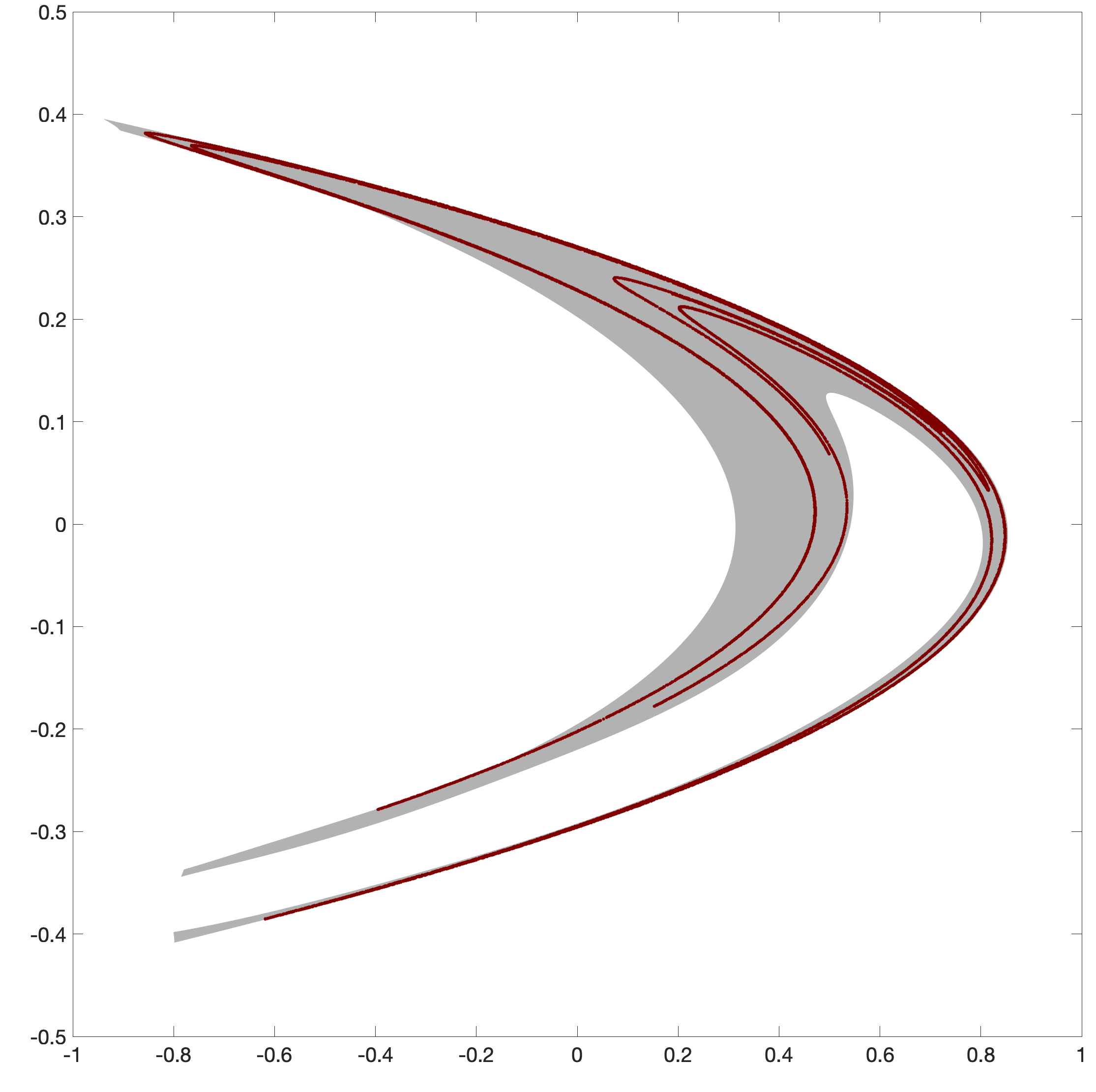}}
\put(250,25){\includegraphics[width=66.4mm]{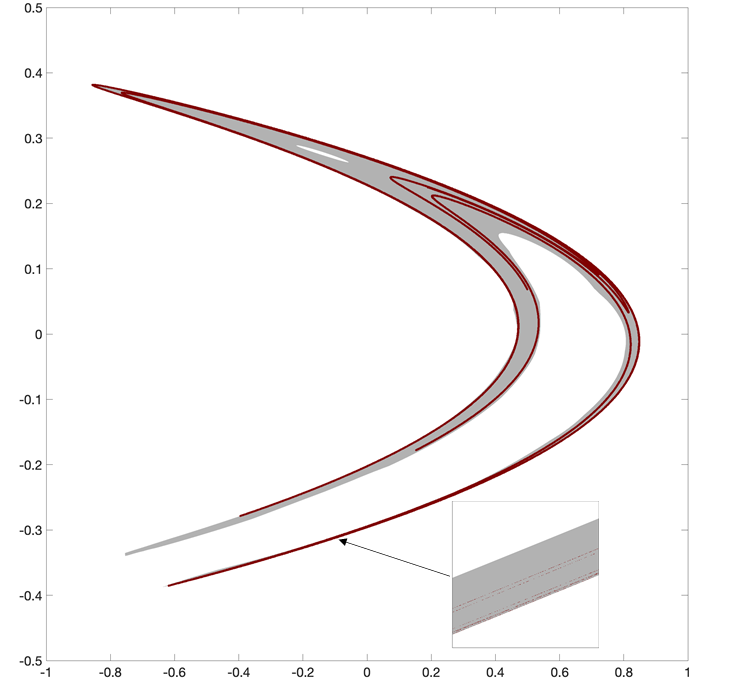}}
\end{picture}
\vspace{-7mm}
\caption{\footnotesize{Outer approximations for the H\'{e}non attractor. Left: $\alpha = 0.05$ and degree 8 polynomials. Right: The figure shows the intersection of the approximations obtained by degree 6,8 and 10 polynomials, all with $\alpha = 0.05$}.}
\label{figHenon}
\end{figure*}

In general, we expect that more complicated topological structures, such as holes, require higher degree polynomials to be identified by our approach and further, since we only gave a guaranteed convergence in terms of the Lebesgue measure, we may not have full control of all topological properties of the outer approximations of the global attractor\footnote{Getting a guaranteed asymptotic control for the topological properties of the attractor would require convergence in the Hausdorff metric. Proving such convergence remains a challenging and so far elusive task for the moment-sum-of-squares approach, here as well as in previous works (e.g.,~\cite{c10neu,c8neu}).}. But we have already seen in the previous figure Fig. \ref{figHenon} that our approach recognizes holes.\\
The Van-der-Pol oscillator is an example where the global attractor is given by an asymptotically stable limit cycle. So the solutions to the SDP tightenings have to detect the limit cycle and hence this is connected to the task of finding holes which we have also seen for the H\'{e}non map.\\
Here it is important to choose the set $X$ a bit more carefully. For the left pane in Figure~(\ref{figVanDerPol}) we chose $X = \{(x_1,x_2) \in \R^2 : 0.4 \leq \|(x_1,x_2)\|_2 \leq 2\}$ so that the limit cycle is included in $X$ but the initial value $(0,0)$ corresponding to the trivial solution $x(t) = y(t) = 0$ for all $t$ is not included. The difference is that if $(0,0)$ is in $X$, then the limit cycle and its whole interior is the global attractor. This is detected by our approach as shown in the right pane of Figure~\ref{figVanDerPol}. The reason why in that case the attractor is the much larger set is that the interior of the limit cycle is the unstable manifold of the equilibrium point $(0,0)$, hence contained in the global attractor.

\begin{figure*}[h]
\begin{picture}(300,210)
\put(250,25){\includegraphics[width=62.5mm]{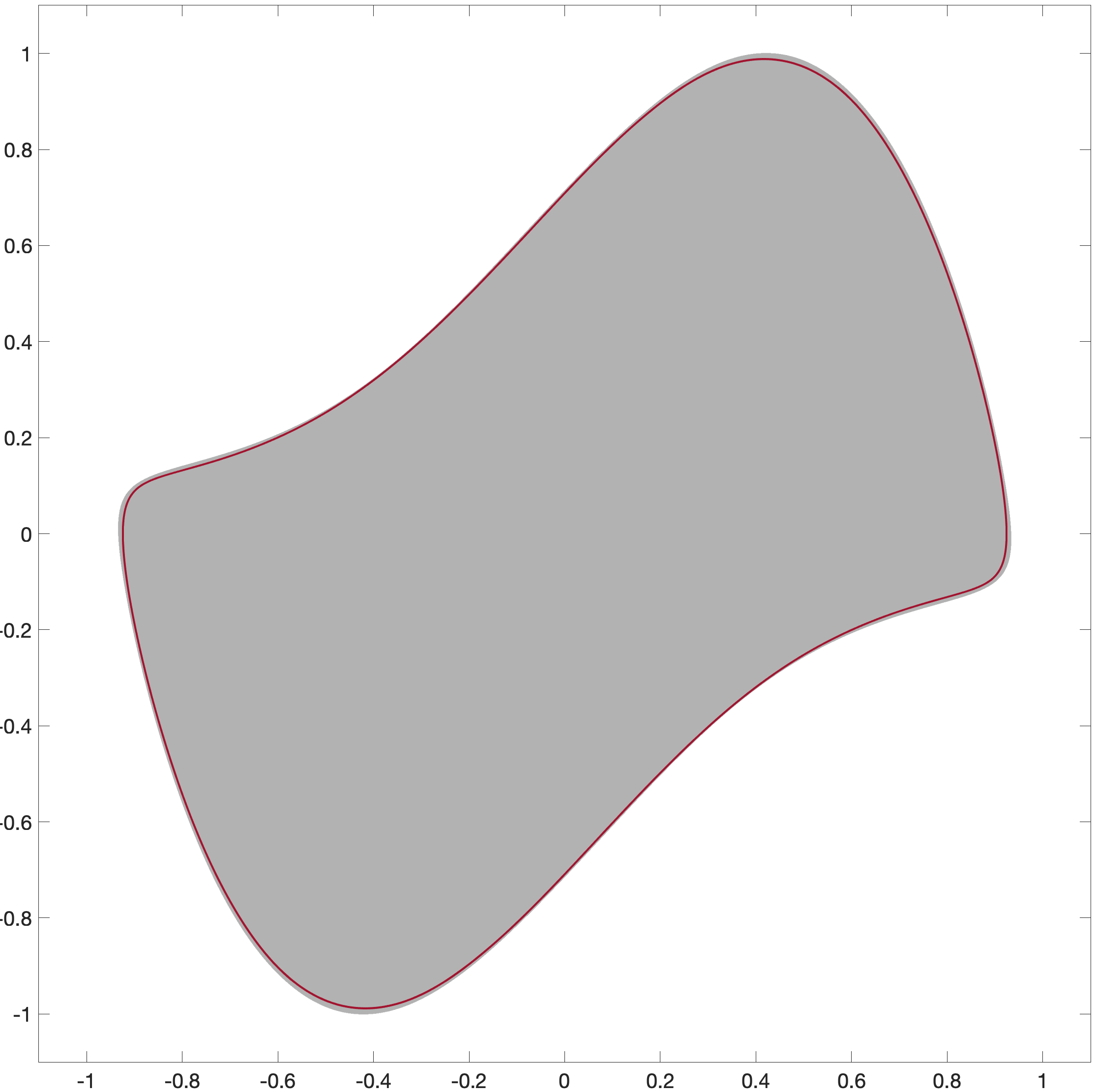}}
\put(30,21){\includegraphics[width=69mm]{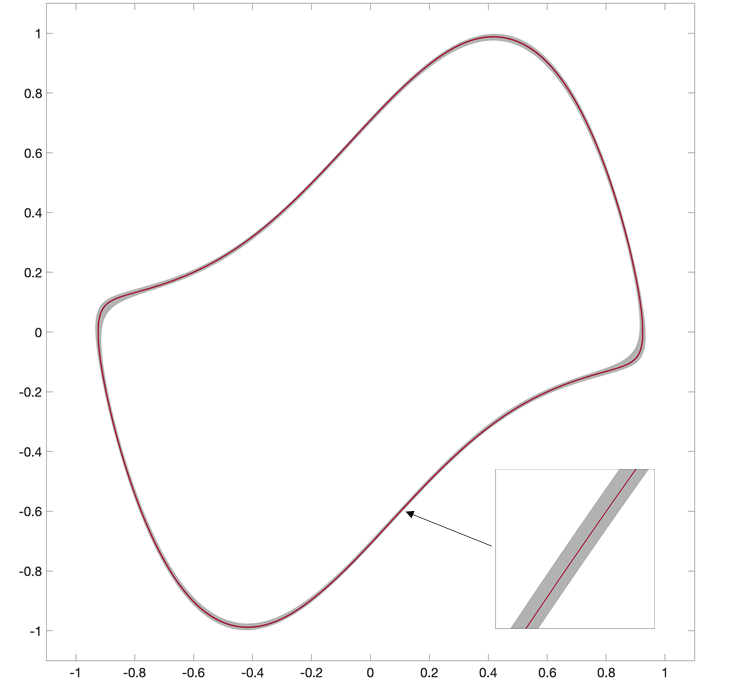}}
\end{picture}
\caption{\footnotesize{Outer approximations of the global attractor for the Van-der-Pol oscillator. Left: approximation with $X = \{x \mid 0.4 \le\|x\|_2 \le 2\}$ (fixed point $(0,0)$ not included), degree 12 polynomials and $\beta = 0.05$. Right: approximation with $X = \{x \mid \|x\|_2 \le 2\}$ (fixed point included), degree 12 polynomials and $\beta = 2$.}}
\label{figVanDerPol}
\end{figure*}

\vspace{-1mm}
The figures in Fig. \ref{figVanDerPol} shows that the global attractor is detected well. But our numerical examples also showed that one has to be careful with numerical issues because, in the case of the Van--der--Pol oscillator for lower degree polynomials, the graph of $\min\{v^1,v^2\}$ is very flat around the global attractor which leads to possible round off issues when depicting the superlevel set $X_k = \min\{v^1,v^2\}^{-1}([0,\infty))$. In such situations, in order to obtain provable outer approximation, a more careful postprocessing of the solutions to the SDP is required, e.g., using the methods of \cite{c19neu,c24neu}.

%


\section{Conclusion}
We presented a linear programming characterization of global attractors as well as a hierarchy of semidefinite programming problems providing asymptotically sharp outer approximations of the global attractor. Therefore, the approach is simple to use, with freely available solvers available both for high-level modeling of the problems (e.g., Yalmip~\cite{c18neu} or Gloptipoly~\cite{c7neu}) and for the numerical solution (e.g., SeDuMi~\cite{c25neu} or MOSEK). In its core, the method builds on and extends the work of~\cite{c10neu} from invariant sets to global attractors, providing explicit constructions and proofs in our setting as well as treating the more subtle discrete time case.

A price for the linear structure or semidefinite structure respectively lies in the dimension of the problem. The size of the largest block of the SDP relaxations scales as $\mathcal{O}(\binom{n+k/2}{n})$ where $n$ is the state space dimension and $k$ the degree bound for the polynomials in the SDP relaxations. A possible approach to tackle this problem is to exploit symmetries or sparsity of the problem. While symmetry exploitation comes at no cost of accuracy~\cite{Symmetry}, obtaining a lossless (or at least convergence-preserving) sparse relaxation in this dynamical context is currently an open challenge (see~\cite{SparseRoA} for results in this direction).\\
Although there are further interesting topological properties to the global attractor some of them are invisible to our approach due to the fact that the primal linear program we use can only identify the global attractor up to a set of Lebesgue measure zero which excludes that we have control of certain topological properties. On the other hand, the dual problem is defined on continuous functions and hence allows more topological insights.\\
A future perspective might be proving convergence without the injectivity assumption for discrete systems. In addition, it is an interesting open question whether global attractors of partial differential equations could be approached in this way. A preliminary work on analysis and control gives promise~\cite{c11neu,c21neu}, although tractably characterizing subsets of the infinite-dimensional state space is likely to be more challenging.

\section{Acknowledgements}
The second author would like to thank David Goluskin for helpful discussions about this work. The authors want to thank Matthew D. Kvalheim for interesting suggestions for the paper and references.

This work has been supported by the European Union's Horizon 2020 research and innovation programme under the Marie Sk\l{}odowska-Curie Actions, grant agreement 813211 (POEMA).

\end{document}